\documentclass[reqno]{amsart}
\usepackage{amsmath}
\usepackage{amssymb}
\usepackage{amsthm}
\usepackage{enumerate}
\usepackage{mathrsfs}
\usepackage{eqlist}
\usepackage{array}

\newtheorem{theorem}{\bf Theorem}[section]

\begin{document}
\title[On locally finite orthomodular lattices]{On locally finite orthomodular lattices}
\author[Dominika Bure\v{s}ov\'{a} \and Pavel Pt\'{a}k]{Dominika Bure\v{s}ov\'{a}* \and Pavel Pt\'{a}k**}

\newcommand{\acr}{\newline\indent}
\newcommand{\N}{\mathbb{N}}
\newcommand{\R}{\mathbb{R}}

\address{\llap{*\,}
Czech Technical University in Prague\acr
Faculty of Electrical Engineering\acr
Prague\acr
Czech Republic}
\email{buresdo2@fel.cvut.cz}
\address{\llap{**\,}Czech Technical University in Prague\acr
Faculty of Electrical Engineering\acr
Department of Mathematics\acr
Prague\acr
Czech Republic}
\email{ptak@math.feld.cvut.cz}

\thanks{This work was supported by FWF: Project I 4579-N and GA\v{C}R: Project 20-09869L} 

\begin{abstract}
    Let us denote by $\mathcal{LF}$ the class of all orthomodular		 lattices (OMLs) that are locally finite (i.e., $L \in \mathcal{LF}$ provided each finite subset of $L$ generates in $L$ a finite subOML). We first show in this note how one can obtain new locally finite OMLs from the initial ones and enlarge thus the class  $\mathcal{LF}$. We find $\mathcal{LF}$ considerably large though, obviously, not all OMLs belong to $\mathcal{LF}$. We then study states on the OMLs of $\mathcal{LF}$. We show that local finiteness may to a certain extent make up for distributivity. We for instance show that if  $L \in \mathcal{LF}$ and if for any finite subOML $K$ there is a state $s: K\to \enskip[0,1]$  on $K$, then there is a state on the entire $L$.
We also consider further algebraic and state properties of $\mathcal{LF}$ relevant to quantum logic theory.
\end{abstract}

\maketitle

\section{Notions and results}

Let $L = (P,\,\leq,\, ',\, 0,\, 1)$ be an orthomodular lattice (abbr. OML - see e.g.~\cite{Kalmbach} for the definition and basic properties of OMLs). As known, an OML does not have to be locally finite. For instance, the free OML over three generators is not locally finite (this OML is infinite, see~\cite{Harding}). Also, the lattice $L(H)$ of the projectors in a Hilbert space $\mathcal{H}$ is not locally finite (see e.g.~\cite{Kalmbach}, where one can also learn basic properties of OMLs). Let us denote by $\mathcal{LF}$  the class of all locally finite OMLs. The basic information of $\mathcal{LF}$  is listed below.

\begin{theorem}\label{Theorem 1.1}\mbox{}\\
	(i) Each finite OML and each Boolean algebra belongs to $\mathcal{LF}$.\\
    (ii) $\mathcal{LF}$ is closed under the formation of finite cartesian products. A consequence: If $L\in \mathcal{LF}$ and if we denote by $cf(L)$ the difference between the number of generators of $L$ and the cardinality of $L$, then $cf(L)$ can in general be greater than any natural number even for a finite $L$. ($\mathcal{LF}$ is not closed under the formation of infinite cartesian products as we shall comment on later.)\\
    (iii) $\mathcal{LF}$  is closed under the formation of arbitrary horizontal sums. A consequence: There is an infinite non-Boolean set-representable locally finite OML.\\
    (iv) $\mathcal{LF}$  is closed under the formation of epimorphic images. A consequence: The class $\mathcal{LF}$ for the objects and the epimorphic homomorphisms for the morphisms forms a category.\\
    (v) Let $L \in \mathcal{LF}$ and let $B$ be a Boolean algebra. Then $L$ can be OML-embedded into an OML $K$, $K \in \mathcal{LF}$ such that the centre $C(K)$ of $K$ is equal to $B$.
\end{theorem}

\begin{proof}
    The statements (i) and (ii) are obvious, a finitely generated Boolean algebra possesses finitely many atoms.
    
\    As for (iii), recall (see~\cite{Kalmbach}) that if $L_\alpha ,\,\alpha \in I$, is a collection of OMLs, then by the horizontal sum $Hor(L_\alpha,\,\alpha \in I)$ one means the OML obtained by identifying  zeros and ones of all $L_\alpha ,\,\alpha \in I$. Obviously, if $L_\alpha \in \mathcal{LF}$ for any $\alpha \in I$, then $Hor(L_\alpha,\, \alpha \in I) \in \mathcal{LF}$, too.

    Further, verifying (iv), let $K \in \mathcal{LF}$ and $e : K \to L$ be an epimorphism in OMLs. Let \{$a_i,\, i \leq n$\} $\subseteq L$.
    Take  $b_i \in K$ with $e(b_i) = a_i$. Then \{$b_i ,\, i \leq n $\}  generates a finite subOML of $K$, some $\tilde{K}$, and so $e(\tilde{K})$ is a finite subOML of $L$. Since $e(\tilde{K})$ contains \{$a_i,\, i \leq n$\}, property (iv) follows.
    
Finally, approaching (v), recall that the centre $C(K)$ means the set of all ``absolutely compatible" elements of $K$ (see~\cite{Kalmbach}). Let us first take the horizontal sum $Hor(L,\, MO_2)$, where $MO_2 = \{0,\, 1,\, a,\, a',\, b,\, b'\}$. Then the centre $C(Hor(L,\, MO_2))$ is trivial, $C(Hor(L,\, MO_2)) = \{0,\, 1\}$, and, obviously, $L$ can be  embedded in $Hor(L,\, MO_2)$. To complete the argument, let us form the so-called Boolean sum of $Hor(L,\, MO_2)$ and $B$ (see e.g.~\cite{Ptak1}). Recall the construction of Boolean sum, the rest is then easily seen. Let us view $B$ as a collection $\mathcal{B}$ of subsets of a set, $B = (D,\,\mathcal{B})$ (the Stone Theorem). Let $\mathcal{P}$ be the collection of all partitions of $B$ (by a partition of $B$ we mean a family $P= \{A_i| A_i \in B,\, i \leq n$\} of disjoint non-empty sets $A_i$ with $\bigcup\limits_{i\leq n}A_i = D$).\\
If $L$ is an OML and $B= (D,\,\mathcal{B})$ is a Boolean algebra, Boolean sum of $L$ and  $B$ is the set of all functions $f :  D \to L$ with the property that there is a partition  \{$A_i,\, i \leq n$\} of $B$ such that $f$ is a constant on any $A_i ,\, i \leq n$ (the ``B step functions" on $D$ with values in $L$). Obviously, Boolean sum of $Hor(L,\, MO_2)$ and $B$ satisfies property (v).
\end{proof}

Let us pass to the considerations concerning states. Let us recall that
$s : L\to$  $[0,1]$ (resp. $s: L \to $ \{$0,1$\}) is called a state on $L$ (resp. a two-valued state on $L$), if\\
(i) $s(1) = 1$, and\\
(ii) $s(a\cup b) = s(a)+s(b)$ provided $a \leq b'$.\acr

Denote by $\mathcal{S}(L)$ (resp. $\mathcal{S}_2 (L)$) the set of all states on $L$ (resp. the set of all two-valued states on $L$). The characterization of the state space of an OML has been obtained in~\cite{Schultz} (for the interplay of $\mathcal{S}(L)$ with algebraic properties of $L$ implicitely related to this paper, see~\cite{NavaraPtakRogalewicz}). A peculiar circumstance is that $\mathcal{S}(L)$ may be empty (see~\cite{Greechie}) that disqualifies certain OMLs for an interpretation as quantum logics (see~\cite{PtakPulmannova} and~\cite{DvurcenskijPulmannova}). For a locally finite $L$, if this phenomenon occurs (i.e, if $\mathcal{S}(L) = \emptyset$), it has to occur for a finite subOML of $L$. This is shown in our next result.
(It may be noted that in the ``hidden variable" interpretation (see~\cite{Gudder}), if there is no hidden variable on $L$ (i.e., if $\mathcal{S}_2(L) = \emptyset$), then there is no hidden variable on a certain finite subOML of $L$).

\begin{theorem}\label{Theorem 1.2}
Let $L \in \mathcal{LF}$.
(i) If $\mathcal{S}(K) \neq \emptyset$ (resp. $\mathcal{S}_2(K) \neq \emptyset$) for any finite subOML $K$ of $L$, then $\mathcal{S}(L) \neq \emptyset$  (resp. $\mathcal{S}_2(L) \neq \emptyset$).\\
(ii) (A compactness theorem)
Let $P \subseteq L$ and let $q \in L$. Suppose that for any subOML $K$ of $L$ that is generated by a finite subset $F$ of $P$ there is a state $\tilde{s} \in \mathcal{S}_2(L)$ with $\tilde{s}(q)=1$ and $\tilde{s}(k)=0$ for any $k \in F$.
Then there is a state $s \in \mathcal{S}_2(L)$ with $s(q)=1$ and $s(p)=0$ for any $p \in P$. 
\end{theorem}

\begin{proof}
(i) By a standard application of Tychonoff's theorem on the product of compact topological spaces, the set $[0,1]^L$ (resp.\{$0,1$\}$^L$) of all functions $f: L \to$  $[0,1]$ (resp. all functions $L \to $ \{$0,1$\}) is compact in the topology of pointwise convergence. Let us consider, for each finite subOML $K$ of $L$, the set of all functions $g:L \to$ $ [0,1]$ that are states on $K$ when restricted to $K$. Denote this set of functions by $\mathcal{T}(K)$. Obviously, each $\mathcal{T}(K)$ is compact, too. Since the family of all finite subOMLs of $L$ is directed when ordered by inclusion, we see that \{$\mathcal{T}(K),\, K $ finite subOML of $L$\} is a centered family of closed subsets of $[0,1]^L$. By our assumption, each $\mathcal{T}(K)$ is non-void. If $\mathcal{K}$ denotes the set of all finite subOMLs of $L$, then the intersection $\bigcap\limits_{K \in \mathcal{K}}\mathcal{T}(K)$ is non-void (the compactness of $[0,1]^L$ applies).
Let $ s \in \bigcap\limits_{K \in \mathcal{K}} \mathcal{T}(K)$. Then $s \in \mathcal{S}(L)$ and this completes the proof.
\\
(ii) Suppose that $F$ is a finite subset of $P$ and suppose that $K$ is generated by $F$. Let us set $\mathcal{V}(K)$ for all states $\mathcal{S}_2(L)$ with $f(q) = 1$ and $f(k) = 0$ for any $k \in F$. Then \{$\mathcal{V}(K)$, $K$ generated by a finite subset of $P$\} is a centered family.
By applying the previous approach of ``let the compactness do it", we obtain a state $s \in \mathcal{S}_2(L)$ with the required properties.

\end{proof}

As expected, the local finiteness is essential in the above results.

\begin{theorem}\label{Theorem 1.3}
Let $L(\R^3)$ be the OML of all projections (linear subspaces) of $\R^3$. Then $\mathcal{S}_2(L(\R^3))$ = $\emptyset$ but each finite subOML $K$ of $L(\R^3)$ is set-representable (and therefore $K$ possesses an order determining set of two-valued states).
\end{theorem}
\begin{proof}
It is well known (see~\cite{KochenSpecker}, for a stronger version see~\cite{VoracekNavara}) that $\mathcal{S}_2(L(\R^3)) = \emptyset$. On the other hand, if $K$ is a subOML of $L(\R^3)$ and $K$ contains $3$ atoms (three one-dimensional subspaces) that are in a general configuration, meaning they are linearly independent and no two atoms are orthogonal, then $K$ is infinite. This can be easily calculated (or obtained as a consequence of the classic result by A.F. M{\"o}bius - $K$ has so many directions that their normalizations form a dense subset of the unit sphere in $\R^3$). The only essentially different situation to be argued is the case when $K$ is generated by $3$ atoms $a_1$, $a_2$ and $a_3$ such that $a_1$ is orthogonal to the plane $a_2 \lor a_3$.The OML $K$ is then generated by $a_2$ and $a_3$. Then $K$ can be viewed as an epimorphic image of the free OML over two generators. The latter free OML is known to be set-representable and it consists of 96 elements (see e.g.~\cite{Kalmbach} and~\cite{Harding}).
Since any epimorphic image of a set-representable OML is again set-representable (see e.g.~\cite{Mayet}), $K$ is set-representable and therefore $K$ contains an order determining set of two-valued states (see e.g.~\cite{Ptak2}). This concludes the proof.
\end{proof}

It may be worth making a short comment, as a fact of a methodological matter, that the local finiteness approach allows one to have less standard proofs of the following classic results for Boolean algebras.
\begin{theorem}\label{Theorem 1.4}
Let $B$ a Boolean algebra.\\
(i) If $A$ is a Boolean subalgebra of $B$ and $s \in \mathcal{S}(A)$, then $s$ can be extended to $B$ as a state.\\
(ii) If $a \neq 0$, then there is such a prime ideal $I$ on $B$ that $a \in I$.
\end{theorem}
\begin{proof}
(i) (See e.g.~\cite{HornTarski}, among other proofs the Krein-Millman theorem provides an apparently new proof strategy, see~\cite{BuresovaPtak}). Let $s \in \mathcal{S}(A)$. For each finite Boolean subalgebra $K$ of $A$ it is easy to construct a state $t: B \to$ $[0,1]$ that agrees with $s$ on $K$. So the (compact) set $\mathcal{U}(K)= \{t \in \mathcal{S}(B),\, t=s \text{ on } K\}$ is non-void. If we denote by $\mathcal{K}$ the set of all finite Boolean subalgebras of $A$, we see that $\bigcap\limits_{K \in \mathcal{K}}\mathcal{U}(K)$ is non-empty. Take $v \in \bigcap\limits_{K \in \mathcal{K}}\mathcal{U}(K)$. Then $v \in \mathcal{S}(B)$ and $v$ extends $s$.

(ii) (The classical Stone Theorem).~Consider the collection $\mathcal{G}$ of all finite Boolean subalgebras of $B$ that contain the element $a$. If $K \in \mathcal{G}$, then there is a two-valued state $\tilde{s} \in \mathcal{S}_2(K)$ such that $\tilde{s}(a)=0$. By the compactness argument used above, there is a state $s \in \mathcal{S}_2(B)$ with $s(a)=0$. Then $I=\{ a \in B |s(a) = 0\}$ is the required ideal.

\end{proof}

In the conclusion, let us ask if $\mathcal{LF}$ contains all set-representable OMLs. Of course, the free OML over three generators is not set-representable. In ~\cite{Harding} the author shows that this OML contains the free OML over countably many generators and if the latter OML were set-representable, then so would be all ``small" OMLs. This is obviously not the case (see~\cite{Greechie}).
However, we can answer the above formulated question to the negative in other ways contributing thus to our theme and adding to \textit{Theorem 1.1}. The first way utilizes the so-called Kalmbach embedding of lattices. Recall that by~\cite{Kalmbach} each lattice $L$ can be lattice-theoretically embedded into an OML. Let us denote it by $K(L)$. It can be shown (see~\cite{HardingNavara}) that $K(L)$ is always set-representable. 
Hence if $L$ is not locally finite as a lattice, and such a lattice is easy to find, then $K(L)$  is a set-representable OML that is not locally finite. Another example related to \textit{Theorem~\ref{Theorem 1.1}(ii)} could be obtained by showing that the class $\mathcal{LF}$ is not closed under the formation of the product of countably many (finite set-representable) OMLs. Indeed, it suffices to find finite set-representable OMLs $L_n$, $n \in \N$, such that $L_n$ is generated by three elements $a_1^n,a_2^n$ and $a_3^n$ and, moreover, the cardinality of $L_n$ is greater than or equal to $n$ (one can for instance use the Kalmbach embedding again). Let us take the Cartesian product $\prod\limits_{n \in \N}L_n$. Then the triple $b_i \in \prod\limits_{n \in \N}L_n,\, i \leq 3$, defined by $b_1=(a_1^n),\ b_2=(a_2^n),\ b_3=(a_3^n),\, n \in \N$, generates an infinite subset of $\prod\limits_{n \in \N}L_n$ (the $n$-th coordinate ``supplies" at least $n$ elements). Since the product of set-representable OMLs is again set-representable (see e.g.~\cite{Mayet}), we see that the product $\prod\limits_{n \in \N}L_n$ is the example we looked for.\acr
In the final note, let us observe that the analogous questions on the local finiteness are pertinent for the recently studied orthomodular structures endowed with an intrinsic symmetric difference. The lattices of this class are denoted by ODLs (the OMLs with an ``XOR" operation;  see~\cite{BikchentaevNavara},~\cite{BuresovaPtak},~\cite{Matousek},~\cite{MatousekPtak},~\cite{Su}, etc.). The results similar to our findings of \textit{Theorem~\ref{Theorem 1.1}} and \textit{Theorem~\ref{Theorem 1.2}} can be also  proved for ODLs. A remarkable result related to the previous paragraph has already been obtained in~\cite{Su}: There is an infinite set-representable ODL that is generated by finitely many elements. So the ODLs do not have to be locally finite even for the set-representable ones.

\end{document}